\documentclass[reqno]{amsart}
\usepackage{graphicx}
\usepackage{geometry}
\usepackage{amsmath}
\usepackage{amscd}
\usepackage{amsthm}
\usepackage{amsfonts}
\usepackage{amssymb}
\usepackage[colorlinks=true]{hyperref}
\newtheorem{theorem}{Theorem}[section]
\newtheorem{corollary}[theorem]{Corollary}
\newtheorem{proposition}[theorem]{Proposition}
\newtheorem{lemma}[theorem]{Lemma}

\usepackage{ifthen, xcolor}
\newboolean{commentBoolVar}
\setboolean{commentBoolVar}{true} 
\newcommand{\colourcomment}[3]
{%
\ifthenelse{\boolean{commentBoolVar}}{{\color{#2}(#1: #3)}}{}%
}%

\theoremstyle{definition}
\newtheorem{definition}[theorem]{Definition}
\newtheorem{example}[theorem]{Example}

\DeclareMathOperator{\Gal}{Gal}
\DeclareMathOperator{\Perm}{Perm}

\newcommand{\gen}[1]{\langle #1 \rangle} 

\newcommand{\B}{\mathfrak{B}}

\newcommand{\pr}{\mathrm{pr}}

\begin{document}

\title{Opposite Skew Left Braces and Applications}
\author{Alan Koch}
\address{Department of Mathematics, Agnes Scott College, 141 E. College Ave., Decatur, GA 30030 USA}
\email{akoch@agnesscott.edu}
\author{Paul J.~Truman}
\thanks{The second author is supported in part by the London Mathematical Society Grant \#41847}
\address{School of Computing and Mathematics, Keele University, Staffordshire, ST5 5BG, UK}
\email{P.J.Truman@Keele.ac.uk}
\date{\today        }

\begin{abstract}
	Given a skew left brace $\B$, we introduce the notion of an ``opposite" skew left brace $\B'$, which is closely related to the concept of the opposite of a group, and provide several applications. Skew left braces are closely linked with both solutions to the Yang-Baxter Equation and Hopf-Galois structures on Galois field extensions. We show that the set-theoretic solution to the YBE given by $\B'$ is the inverse to the solution given by $\B$; this allows us to identify the group-like elements in the Hopf algebra providing the Hopf-Galois structure using only these solutions. We also show how left ideals of $\B'$ correspond to the realizable intermediate fields of a certain Hopf-Galois extension of a Galois extension. 
\end{abstract}

\maketitle

\section{Introduction}

Skew left braces were developed by Guarnieri and Vendramin in \cite{GuarnieriVendramin17} to construct non-degenerate, not necessarily involutive set-theoretic solutions to the Yang-Baxter Equation. They were developed as a generalization to the concept of braces defined by Rump in \cite{Rump07} to find involutive solutions to the YBE. As first pointed out in \cite[Remark 2.6]{Bachiller16} and developed in the appendix by Byott and Vendramin in \cite{SmoktunowiczVendramin18}, finite skew left braces--hereafter, ``braces" for brevity--always arise from Hopf-Galois structures on Galois field extensions. In \cite[Remark 2.6]{Bachiller16}, the author writes ``We hope that this connection between these two theories would be fruitful in the future", a hope which has been fulfilled: for example, in \cite{Childs18} Childs defines the notion of ``circle-stable subgroups'' of a brace and shows that such subgroups correspond to sub-Hopf algebras of the Hopf algebra giving the corresponding Hopf-Galois structure. 

In this work (see Section \ref{opp}), we introduce the rather simple notion of the opposite of a skew left brace. Our construction simply reverses the order in one of the two binary operations which determine the brace. Our motivation comes from an existing pairing of non-commutative Hopf-Galois structures. We illustrate the usefulness of the opposite construction through a few applications.

As mentioned above, skew left braces provide set-theoretic solutions to the Yang-Baxter equation which are non-degenerate. Given a set $B$, a solution is a function $R:B\times B\to B\times B$ satisfying certain properties--see Section \ref{YBE} for details. Each brace $\B$ gives rise to such a solution $R_{\B}$: the non-degeneracy of $R_{\B}$ implies that it has an inverse; in Section \ref{YBE2} we show how the opposite brace allows for an easy construction of $R_{\B}^{-1}$.

Suppose $L/K$ is a finite Galois extension. Then Hopf-Galois structures on $L/K$ correspond with choices of certain groups $N$ of permutations of the elements of $\Gal(L/K)$, which in turn give rise to braces $\B(N)$. Unlike classical Galois theory, a Hopf-Galois structure will give some, but not necessarily all, intermediate fields of $L/K$, only the ones which correspond to sub-Hopf algebras. It is natural to ask which intermediate fields arise, which \cite{Childs18} answers by constructing a new substructure of a brace. In Section \ref{HGS} we use the opposite and relate these intermediate fields with the known brace substructure of ideals (and the closely related, new concept of quasi-ideals) of the opposite brace. Ideals allow us not only to find these intermediate fields $K\le F \le L$, but also single out, for example, which allow $L$ to be decomposed into two Galois extensions $L/F$ and $F/K$ which are also Hopf-Galois in a manner canonically related to the original Hopf-Galois structure.

One can also use the constructed solution to the YBE to understand some of the structure of the Hopf algebra which provides the corresponding Hopf-Galois structure. By using both connections to skew left braces, we are able to determine the group-like elements of the Hopf algebra by examining the second component of the solution to the YBE.

It is possible that a brace be equal to its own opposite, but it is easy to see that this happens if and only if a certain commutativity condition is satisfied. However, it is also possible to have a brace be isomorphic to its opposite, forming what we call, with abuse of terminology, a {\it self-opposite} brace. Knowing if a brace is self-opposite has important consequences when determining the intermediate fields in a Hopf-Galois extension which arise through the Hopf-Galois correspondence. Thus, in Section \ref{SO} we consider the self-opposite question. At this point, there seems to be no simple criterion to determine whether a brace is self-opposite.

The first author would like to thank Keele University for its hospitality during the development of this paper.

\section{(Skew Left) Braces, the Yang-Baxter Equation, and Hopf-Galois Structures}

In this section we provide the background necessary for the rest of the paper. 
 
\subsection{Braces} We begin, of course, with the definition of a skew left brace. At this point, there does not seem to be standard notation for skew left braces; we set ours based on \cite{GuarnieriVendramin17}.

\begin{definition}
A {\it skew left brace} $\B$ is a triple $(B,\cdot,\circ)$ consisting of a set and two binary operations, where $(B,\cdot)$ and $(B,\circ)$ are both groups and the following relation holds for all $x,y,z\in B$:
\[x\circ(yz) = (x\circ y)\cdot x^{-1}\cdot  (x\circ z),\]
where the symbol $x^{-1}$ refers to the inverse to $x\in (B,\cdot)$. We call the relation above the {\it brace relation}.
\end{definition}

As one would expect, a {\it brace homomorphism} is a map preserving both the dot and circle operations, and an bijective homomorphism is a {\it brace isomorphism}.

As stated in the introdution, for brevity we will refer to a skew left brace simply as a {\it brace}, however the reader should be aware that ``(left) brace'' is used by many to refer to the case where $(B,\cdot)$ is abelian as in \cite{Rump07}.

Going forward, we will adopt the following notational conventions for $\B=(B,\cdot,\circ)$, the first (mentioned above) included for completeness:
\begin{itemize}
	\item For $x\in B$, the inverse to $x$ in $(B,\cdot)$ will be denoted $x^{-1}$.
	\item For $x\in B$, the inverse to $x$ in $(B,\circ)$ will be denoted $\overline{x}$.
	\item For $x,y\in B$ we will write $xy$ for $x\cdot y$ when no confusion can arise. 
	\item The identity in both $(B,\cdot)$ and $(B,\circ)$ will be denoted $1_B$. Note that the symbol $1_B$ is not ambiguous: if $x\cdot 1_B= x$ for all $x\in B$ then
	\[x\circ 1_B = x\circ (1_B\cdot 1_B) = (x\circ 1_B)x^{-1}(x\circ 1_B),\]
	from which it follows from left cancellation that $x^{-1}(x\circ 1_B) = 1_B$, i.e., $x=x\circ 1_B$.
\end{itemize}
Here are some examples which will be used throughout this paper.

\begin{example}\label{one}
	Let $(B,\cdot)$ be any finite group. Then $\B=(B,\cdot,\cdot)$ is readily seen to be a brace. We call this the {\it trivial brace} on $B$.
\end{example}

\begin{example}\label{two}
	Let $(B,\cdot)$ be any finite group, and define $x\circ y = y\cdot x$ for all $x,y\in B$. Then $\B=(B,\cdot,\circ)$ is also a brace. We call this the {\it almost trivial brace} on $B$.
\end{example}

\begin{example}\label{three}
	Let
	\[B=\gen{\eta,\pi:\eta^4=\pi^2=\eta\pi\eta\pi=1}.\]
	Then $B\cong D_4$, the dihedral group of order $8$. Define a binary operation $\circ$ on $B$ as follows:
	\[\eta^i\pi^j\circ \eta^k\pi^{\ell} = \eta^{2j{\ell}}\left(\eta^k\pi^{\ell}\right)\left(\eta^i\pi^{j}\right)= \eta^{k+(-1)^{\ell}i+2j{\ell}}\pi^{j+\ell}.\]
	Note that $\eta^{2j\ell}$ is in the center of $(B,\cdot)$. This operation is associative: let $x_j=\eta^i\pi^j,\;x_{\ell}=\eta^k\pi^{\ell},$ and $x_n=\eta^m\pi^n$ for some choices $i,k,m$. Observe that, e.g., $x_jx_{\ell}=y_{j+\ell}$ for some $y_{j+\ell} = \eta^r\pi^{j+\ell}$. Then
	\[ x_j\circ(x_{\ell}\circ x_n) = x_j \circ (\eta^{2\ell n}x_nx_{\ell}) = \eta^{2j(\ell+n)}\eta^{2\ell n}x_nx_{\ell}x_j=\eta^{2j\ell +2jn + 2\ell n}x_nx_{\ell}x_j\]
	and similarly
	\[(x_j\circ x_{\ell})\circ x_n = \eta^{2j{\ell}}x_{\ell}x_j\circ x_n = \eta^{2n(j+{\ell})}\eta^{2j{\ell}}x_nx_{\ell}x_j=\eta^{2j\ell +2jn + 2\ell n}x_nx_{\ell}x_j.\]
	Additionally, $\eta^i\pi^j\circ 1_B = \eta^i\pi^j$ so $1_B$ is the identity, and 
	\[\eta^i\circ \eta^{-i} = 1_B,\; \eta^{i}\pi \circ {\eta^{i+2}\pi} = 1_B\]
	shows $\overline{\eta^i}=\eta^{-i}$ and $\overline{\eta^i\pi} = \eta^{i+2}\pi$ and hence $(B,\circ)$ is a group. The identities $\pi\circ \eta = \eta^{-1}\circ \pi$ and $\pi\circ \pi =\eta \circ \eta$ can be easily established, and since 
	\[\eta^{\circ k} := \underbrace{\eta\circ\eta \cdots \circ \eta}_{k\text{ times}} = \eta^k\]
	we see $\eta\in(B,\circ)$ has order $4$, hence $(B,\circ)\cong Q_8$.
	
	Finally, we claim that $(B,\cdot\circ)$ satisfies the brace relation. Writing $x_j,x_{\ell}$, and $x_n$ as above we get
	\begin{align*}
	x_j\circ (x_{\ell} x_n)
	&= \eta^{2j(\ell+n)} x_{\ell}x_nx_j\\
	&= \eta^{2j\ell} x_{\ell} \eta^{2jn}x_nx_j\\
	&= \eta^{2j{\ell}}x_{\ell}(x_j x_j^{-1}) \eta^{2jn}x_{n}x_j\\
	&= (x_j\circ x_{\ell})x_j^{-1}(x_j \circ x_n),
	\end{align*}
	and hence $(B,\cdot,\circ)$ is a brace.
\end{example}

\subsection{The Yang-Baxter Equation}\label{YBE}

As mentioned previously, skew left braces were originally constructed to provide set-theoretic solutions to the Yang-Baxter Equation. We now review this concept.

\begin{definition} A {\it set-theoretic solution to the Yang-Baxter Equation} is a set $B$ together with a function $R:B\times B \to B\times B$ such that 
	\[R_{12}R_{23}R_{12}(x,y,z) = R_{23}R_{12}R_{23}(x,y,z)\]
for all $x,y,z\in B$, where $R_{12} = R\times 1_B$ and $R_{23}=1_B\times R$. 

Furthermore, we say $R$ is {\it involutive} if $R(R(x,y))=(x,y)$ for all $x,y\in B$; and if we write $R(x,y)=(f_x(y),f_y(x))$ for some functions $f_x,f_y:B\to B$ we say $R$ is {\it non-degenerate} if $f_x$ and $f_y$ are both bijections.
\end{definition}

Notice above that we will often refer to $R$ as the solution, leaving $B$ implicit.

\begin{example}\label{oneYBE}
	Let $B$ be any finite group, written multiplicatively. Then $R(x,y) = (y,y^{-1}xy)$ is a non-degenerate solution to the YBE. It is involutive if and only if $B$ is abelian.
\end{example}

\begin{example}\label{twoYBE}
	In a manner similar to the above, let $B$ be any finite group, written multiplicatively. Then $R(x,y) = (x^{-1}yx,x)$ is a non-degenerate solution to the YBE. It is also involutive if and only if $B$ is abelian.
\end{example}

\begin{example}\label{threeYBE}
	Let $B$ be the set $B=\{\eta^i\pi^j:0\le i \le 3,0\le \pi \le 1 \}$.  Then
	\[R(\eta^i\pi^j,\eta^k\pi^{\ell})=\left(\eta^{(-1)^jk+2i\ell+2j\ell}\pi^{\ell},\eta^{i+2j\ell}\pi^{j} \right)\]
		provides a non-degenerate solution to the YBE, where the exponent on $\eta$ is interpreted $\bmod\;4$ and the exponent on $\pi$ is interpreted $\bmod \;2$. We leave the details to the reader for now, although it will follow from the paragraph to follow that $R$ must satisfy with YBE.
\end{example}

The connection between solutions to the YBE and braces are as follows. Suppose $\B:=(B,\cdot,\circ)$ is a brace. Let $R_{\B}:B \times B\to B\times B$ be given by
\[R_{\B}(x,y) = (x^{-1}(x\circ y), \overline{x^{-1}(x\circ y)}\circ x \circ y).\]

By \cite[Theorem 1]{GuarnieriVendramin17}, $R_{\B}$ provides a non-degenerate, set-theoretic solution to the YBE, involutive if and only if $(B,\cdot)$ is abelian. In fact, Examples \ref{oneYBE}, \ref{twoYBE}, \ref{threeYBE} were constructed from the braces given in Examples \ref{one}, \ref{two}, and \ref{three} respectively. 

\subsection{Hopf-Galois Structures}

We start by recalling the definition of a Hopf-Galois extension--more details can be found, e.g., in \cite[\S 2]{Childs00}.

\begin{definition}
Let $L/K$ be a field extension. Suppose there exists a $K$-Hopf algebra $H$, with comultiplication and counit maps $\Delta$ and $\varepsilon$ respectively,  which acts on $L$ such that
\begin{enumerate}
	\item $h\cdot(st) = \operatorname{mult}\Delta(h)(s\otimes t), h\in H,\; s,t\in L$,
	\item $h(1)=\varepsilon(h)1,\;h\in H$,
	\item The $K$-module homomorphism $L\otimes_KH \to \operatorname{End}_K(L)$ given by $(s\otimes h)(t) = sh(t), \; h\in H, s,t\in L$ is an isomorphism.
\end{enumerate}
Then $H$ is said to provide a {\it Hopf-Galois structure} on $L/K$, and we say $L/K$ is Hopf-Galoiswith respect to $H$, or $H$-Galois.

\end{definition}
If $H$ gives a Hopf-Galois structure on $L/K$ then
\[L^H:=\{s\in L:h(s) = \varepsilon(h)s \text{ for all } h\in H\} = K,\]
and we think of $K$ as the ``fixed field'' under this action. If $H_0$ is a sub-Hopf algebra of $H$, then $L^{H_0}$, defined analogously, is an intermediate field in the extension $L/K$. While the usual Galois correspondence provides a bijection between subgroups and intermediate fields, the correspondence between sub-Hopf algebras and intermediate fields need not be onto (though it is certainly injective).

In the groundbreaking paper \cite{GreitherPareigis87}, Greither and Pareigis showed that Hopf-Galois structures on any separable field extension $L/K$ could be found using only group theory; we shall outline their results in the case where $L/K$ is Galois. Let $G=\Gal(L/K)$, and let $\Perm(G)$ denote the group of permutations of $G$. A subgroup $N\le \Perm(G)$ is called {\it regular} if for all $g,h\in G$ there exists a unique $\eta\in N$ such that $\eta[g]=h$. Note that $N$ must have the same order as $G$. Furthermore, we shall say $N$ is {\it $G$-stable} if $^g\eta\in N$ for all $g\in G,\eta\in N$, where $^g\eta\in\Perm(G)$ is given by
\[^g\eta[h] = \lambda(g)\eta\lambda(g^{-1})[h],\;h\in G\]
and $\lambda(k)\in\Perm(G)$ is left multiplication by $k\in G$.

Given a regular, $G$-stable $N\le \Perm(G)$, let $H_N$ be the invariant ring $H_N = L[N]^G$, where $G$ acts on $N$ as above and on $L$ through Galois action. Then $H_N$ is a $K$-Hopf algebra which acts on $\ell\in L$ via 
\[\left(\sum_{\eta \in N} a_\eta \eta \right) \cdot \ell = \sum_{\eta \in N} a_\eta \eta^{-1}[1_G](\ell),\;\sum_{\eta \in N} a_\eta \eta\in H_N\subset L[N]^G.\]
The association $N\mapsto H_N$ for $N\le \Perm(G)$ is a bijection between regular, $G$-stable subgroups and Hopf Galois structures on $L/K$.

\begin{example}\label{twoHGS}
	Let $N=\rho(G)=\{\rho(g):g\in G\}$, where $\rho(g)[h]=hg^{-1}$ is right regular representation. For $h,k\in G$ then $\rho(g)[h]=k$ if and only if $g=hk^{-1}$, hence $\rho(G)$ is regular. Since the images of the left and right regular representations commute, $\rho(G)$ is $G$-stable. In fact, $\lambda(g)$ acts trivially on $\rho(h)$ for all $g,h\in G$, so $H_{\rho(G)}\cong K[G]$. Using the formula given above we see that the action of $ H_{\rho(G)} $ on $ L $ corresponds to the usual action of $ K[G] $, and so we recover the classical Galois structure on $L/K$.
\end{example}

\begin{example}\label{oneHGS}
	Let $N=\lambda(G)=\{\lambda(g):g\in G\}$, $\lambda(g)$ as above. Then $\lambda(G)$ is regular, and since $^g\lambda(h)=\lambda(ghg^{-1})\in \lambda(G)$ we see that $\lambda(G)$ is a $G$-stable subgroup of $\Perm(G)$. The structure given by $H_{\lambda(G)}$ is called the {\it canonical nonclassical Hopf-Galois structure} in \cite{Truman16}.  
\end{example}

\begin{example}\label{threeHGS}
	Suppose $G=\gen{s,t:s^4=t^4=1,s^2=t^2,stst^{-1}=1}\cong Q_8$. Let $\eta=\rho(s),\pi=\lambda(s)\rho(t)\in\Perm(G)$, and let $N=\gen{\eta,\pi}$. Then $N\le\Perm(G)$ is regular, $G$-stable, and $N\cong D_4$, the dihedral group of order $4$: see \cite[Lemma 2.5]{TaylorTruman19} for details. Note that this is one of many regular, $G$-stable subgroups of $\Perm(G)$, as found in {\it loc.\@ cit.}
\end{example}

\subsection{Connecting Braces to Hopf-Galois Structures}\label{conn}

As mentioned in the introduction, Bachiller points out a connection between Hopf-Galois structures and braces. We shall describe this connection using an equivalent, but different, formulation of the correspondence. 

Let $\ast_G$ denote the group operation on some finite group $G$, and suppose $(N,\cdot)\le \Perm(G)$ is regular and $G$-stable. Then there is a map $a:N\to G$ given by
\[a(\eta) = \eta[1_G].\]
By the regularity of $N$, $a$ is a bijection. We define a binary operation $\circ$ on $N$ by
\[\eta \circ\pi = a^{-1}(a(\eta)\ast_G a(\pi)),\;\eta,\pi\in N.\]
Then $(N,\circ)\cong (G,\ast_G)$, and $(N,\cdot,\circ)$ is a brace--note that $ G $-stability is used in verifying that the brace relation holds. We shall denote this brace by $\B(N)$, which we understand depends implictly on $G$. As every Hopf-Galois structure on a Galois extension with group $G$ corresponds to a regular, $G$-stable $N$ we get can construct a brace for every such structure. 

\begin{example}
	Let $N=\rho(G)=\{\rho(g):g\in G\}$, where $\rho(g)[h]=hg^{-1}$ is right regular representation. Then $a:\rho(G)\to G$ is the ``inverse'' map $a(\rho(g))=g^{-1}$, and the corresponding brace has circle operation
	\[\rho(g)\circ \rho(h) = a^{-1}(a(\rho(g))a(\rho(h))) =a^{-1}(g^{-1}h^{-1})= \rho((g^{-1}h^{-1})^{-1}) = \rho(hg) = \rho(h)\rho(g)\]
	giving the almost trivial brace constructed in Example \ref{two}.
\end{example}

\begin{example}
	Let $N=\lambda(G)$, so $a:N\to G$ is simply $a(\lambda(g))=g$. Then
	\[\lambda(g)\circ \lambda(h) = a^{-1}(a(\lambda(g))a(\lambda(h))) = a^{-1}(gh) = \lambda(gh) = \lambda(g)\lambda(h)\]
	giving the trivial brace from Example \ref{one}
\end{example}

\begin{example}\label{threeHGS2B}
	Let $G,N$ be as in Example \ref{threeHGS}. Then $a:N\to G$ is given by 
	\[a(\eta^i)=\eta^{i}[1_G]=s^{-i},\;a(\eta^i\pi) = \eta^i\pi[1_G] = \eta^{i}[st^{-1}]=st^{-1}s^{-i}=s^{i+1}t^{-1}.\]
It is easiest to work out the circle operation in cases, depending on the powers of $\pi$. We have
	\begin{align*}
	\eta^i\circ\eta^k &= a^{-1}(s^{-i-k})=\eta^{i+k} \\
	\eta^i\circ\eta^k\pi &= a^{-1}(s^{-i+k+1}t^{-1}) = \eta^{k-i}\pi\\
	\eta^i\pi\circ\eta^k &= a^{-1}(s^{i+1}t^{-1}s^{-k}) = a^{-1}(s^{i+k+1}t^{-1}) = \eta^{i+k}\pi\\
	\eta^i\pi\circ\eta^k\pi &= a^{-1}(s^{i+1}t^{-1}s^{k+1}t^{-1}) = a^{-1}(s^{i-k}t^{-2}) = a^{-1}(s^{i-k+2})=\eta^{k-i-2}.
	\end{align*}
	Generally, 
	\[\eta^i\pi^j\circ \eta^k\pi^{\ell} = \eta^{k+(-1)^{\ell}i+2j{\ell}}\pi^{j+\ell},\]
	which agrees with the brace constructed in Example \ref{three}.

\end{example}

Conversely, suppose $\B=(B,\cdot,\circ)$ is a brace. Then $(B,\circ)$ is a group. For each $x\in B$ define $\eta_x\in\Perm(B,\circ)$ by 
\[\eta_x[y]=x\cdot y,\; y\in B.\]
Then $\eta_x[y]=z$ if and only if $x=z\cdot y^{-1}$, so $N=\{\eta_x : x\in B\}$ is a regular subgroup of $\Perm(B,\circ)$. Furthermore, $N$ is $(B,\circ)$-stable: for $x,y\in B$ we have, since $\lambda(y)\in\Perm(B,\circ)$ is left multiplication in the circle group,
\begin{align*}
^y\eta_x[z]&=\lambda(y)\eta_x\lambda(\overline{y})[z]\\
&= \lambda(y)\eta_x[\overline{y}\circ z]\\
&= \lambda(y)[x\cdot(\overline{y}\circ z)]\\
&= y\circ (x\cdot(\overline{y}\circ z)) \\
&= (y\circ x) y^{-1} (y\circ \overline{y}\circ z)\\
&= (y\circ x) y^{-1} \cdot z\\
&=\eta_{(y\circ x)y^{-1}}[z],
\end{align*}
so $^y{\eta_x}=\eta_{(y\circ x)y^{-1}}\in N$. Thus, $N$ is a regular, $(B,\circ)$-stable subgroup of $\Perm(B,\circ)$, hence $N$ provides a Hopf-Galois structure on any Galois extension $L/K$ with Galois group isomorphic to $(B,\circ)$.

\begin{example}\label{four}
	Let $G=\gen{s,t}\cong Q_8$ as in Example \ref{threeHGS}. Let $\eta_t=\rho(t),\;\pi_t=\lambda(t)\rho(s)$, and let $N_t=\gen{\eta_t,\pi_t}$. Then, by \cite[Lemma 2.5]{TaylorTruman19}, $N_t\le\Perm(G)$ is regular,  $G$-stable,  isomorphic to $D_4$, but different from the one considered in Example \ref{threeHGS}. Proceeding in a manner similar to \ref{threeHGS2B} one can show
	\[\eta_t^i\pi_t^j\circ \eta_t^k\pi_t^{\ell} = \eta_t^{k+(-1)^{\ell}i+2j{\ell}}\pi_t^{j+\ell},\]
	and thus we see that different Hopf-Galois structures can give the same brace. 
\end{example}

\section{The Opposite Brace}\label{opp}

In this section, we shall define the opposite brace and describe some of its properties.

\begin{proposition}
	Let $\B=(B,\cdot,\circ)$ be a brace, and for each $x,y\in B$ define $x \cdot' y = yx$. Then $\B':=(B,\cdot',\circ)$ is a brace.
\end{proposition}

\begin{proof}
	Clearly, $(B,\circ)$ is a group, and since $(B,\cdot')$ is the opposite group of $(B,\cdot)$ it is a group as well, sharing the same identity and inverses. It remains to show the brace relation. For $x,y,z\in B$ we have, using the brace relation on $\B$,
	\begin{align*}
	x\circ(y\cdot'z) &= x\circ(zy) \\
	&= (x\circ z)x^{-1}(x\circ y) \\
	&= (x\circ z) \cdot \left( (x\circ y)\cdot' x^{-1} \right) \\
	&=  \left( (x\circ y)\cdot' x^{-1} \right) \cdot' (x\circ z) \\
	&= (x\circ y)\cdot' x^{-1} \cdot' (x\circ z), 
	\end{align*}
and hence $\B'$ is a brace.
\end{proof}

\begin{definition}
	For $\B=(B,\cdot,\circ)$ a brace, the brace $\B'$ constructed above is called the {\it opposite} brace to $\B$.
\end{definition}

We list the following properties for future reference. Their proofs are trivial and omitted.

\begin{lemma}
	Let $\B=(B,\cdot,\circ)$ be a brace. Then
	\begin{enumerate}
		\item $(\B')' = \B$.
		\item If $(B,\cdot)$ is abelian, then $\B'=\B$.
		\item If $\mathfrak{C}$ is a brace, and $f:\B\to\mathfrak{C}$ is a brace homomorphism, then $f$ is also a brace homomorphism $\B'\to\mathfrak{C}'$.
	\end{enumerate}
\end{lemma}

Opposite braces arise arise from an existing construction in Hopf-Galois theory, which we term the {\it opposite Hopf Galois structure}, which we shall now describe. 
Let $G$ be a group, and let $N \le \Perm(G)$ be regular and $G$-stable. Define
\[N'=\operatorname{Cent}_{\Perm(G)}(N)=\{\eta'\in \Perm(G) : \eta\eta'=\eta'\eta \text{ for all }\eta \in N\}.\]
Then, by \cite[Lemmas 2.4.1, 2.4.2]{GreitherPareigis87}, $N'$ is a regular, $G$-stable subgroup of $\Perm(G)$. In fact, for $\eta\in N$, define $\phi_{\eta}\in\Perm(G)$ by $\phi_{\eta}[g]=\mu_g[\eta[1_G]]$, where $\mu_g$ is the element of $N$ such that $\mu_g(1)=g$ (such a $\mu_g$ exists, and is unique, by regularity). One can show that $\phi_{\eta}\phi_{\pi} = \phi_{\pi\eta}$ for $\eta,\pi\in N$, and that $N'$ naturally identifies with the opposite group $N^\mathrm{opp}$ of $N$. The relationship between $N$ and $N'$ has been explored in the area of Hopf-Galois module theory, producing some interesting results \cite{Truman18}.  

Let us compute the brace corresponding to $N'$. Let $a:N\to G$ and $a':N'\to G$ be the bijections obtained by evaluation at $1_G$ as before. Then
\[a'(\phi_{\eta}) = \phi_{\eta}[1_G] = \mu_{1_G}[\eta[1_G]] = 1_N[\eta[1_G]] = \eta[1_G] = a(\eta), \]
hence $\B(N')=(N',\cdot,\circ')$ with
\begin{align*}
\phi_{\eta}\circ' \phi_{\pi} &= (a')^{-1}(a'(\phi_{\eta})a'(\phi_{\pi}))\\
&=(a')^{-1}(a(\eta)a(\pi))\\
&=(a')^{-1}aa^{-1}(a(\eta)a(\pi))\\
&= (a')^{-1}a(\eta\circ\pi)\\
&=\phi_{\eta\circ\pi}.
\end{align*}
Define $f:\B(N')\to (\B(N))'$ by $f(\phi_{\eta}) = \eta$ for all $\eta\in N$. Then 
\[f(\phi_{\eta}\circ'\phi_{\pi}) = f(\phi_{\eta\circ \pi}) = \eta\circ\pi = f(\phi_{\eta})\circ f(\phi_{\pi})\]
and 
\[f(\phi_{\eta}\phi_{\pi}) = f(\phi_{\pi\eta}) = \pi\eta = \eta \cdot' \pi = f(\eta)\cdot' f(\pi)\]
for all $\eta,\pi\in N$. Thus:

\begin{proposition}
	With the notation as above, $\B(N')\cong (\B(N))'$.
\end{proposition}

\begin{example}
	Let $G=\gen{s,t:s^4=t^4=1, s^2=t^2,stst^{-1}}\cong Q_8$ as in Example \ref{threeHGS}. In \cite[Lemma 2.5]{TaylorTruman19} one finds six different regular, $G$-stable subgroups which are isomorphic to $D_4$, namely
	\[\begin{array}{c c c}
	N_{s,\rho} = \gen{\rho(s),\lambda(s)\rho(t)} & N_{t,\rho} = \gen{\rho(t),\lambda(t)\rho(s)} & N_{st,\rho} = \gen{\rho(st),\lambda(st)\rho(t)} \\
	N_{s,\lambda} = \gen{\lambda(s),\lambda(t)\rho(s)} & N_{t,\lambda} = \gen{\lambda(t),\lambda(s)\rho(t)} & N_{st,\lambda} = \gen{\lambda(st),\lambda(t)\rho(st)}.
	\end{array}\]
	Note that the first two correspond to Examples \ref{threeHGS} and \ref{four} respectively. We have seen that $\B(N_{s,\rho})\cong \B(N_{t,\rho})$, and it is easy to see that they are isomorphic to $\B(N_{st,\rho})$ as well. One can quickly verify that the elements of $N_{x,\rho}$ and $N_{x,\lambda}$ commute for each $x\in\{s,t,st\}$, hence the three subgroups in the second row all correspond (up to isomorphism) the same brace, namely $\B(N_{s,\rho})'$.
\end{example}

\section{The Inverse Solution to the Yang-Baxter Equation}\label{YBE2}

Earlier, we saw how a brace $\B:=(B,\cdot,\circ)$ provides us with a set-theoretic solution $R_{\B}$ to the YBE: one which is always non-degenerate, and one which is involutive (that is, self-inverse) if and only if $(B,\cdot)$ is abelian. It is natural to wonder what the inverse to $R_{\B}$ is when $(B,\cdot)$ is not abelian. Since $\B=\B'$ if and only if $(B,\cdot)$ is abelian, perhaps the opposite brace can help us determine the inverse. In fact:

\begin{theorem}
	Let $\B$ be a brace with corresponding solution to the Yang-Baxter Equation $R_{\B}$. Then $R_{\B'}$ is a two-sided inverse to $R_{\B}$, that is, $R_{\B'}R_{\B}(x,y)=R_{\B}R_{\B'}(x,y)=(x,y)$ for all $x,y\in B$. 
\end{theorem}

\begin{proof}
	By interchanging a brace with its opposite, it suffices to show that $R_{\B'}R_{\B}(x,y)=(x,y)$ for all $x,y\in B$. Recall that both $\B$ and $\B'$ have the same inverses, i.e., $x\cdot' x^{-1} = x\circ \overline{x} = 1_B$ where $x^{-1},\overline{x}$ are the inverses in $\B$. 
	
	Let $x,y\in B$. We have
	\begin{align*}
	R_{\B}(x,y) &= (x^{-1}\cdot(x\circ y), \overline{x^{-1}\cdot(x\circ y)}\circ x \circ y)\\
	R_{\B'}(x,y) &= (x^{-1}\cdot'(x\circ y), \overline{x^{-1}\cdot'(x\circ y)}\circ x \circ y)
	= ((x\circ y)\cdot x^{-1}, \overline{(x\circ y)\cdot x^{-1}}\circ x \circ y)
	\end{align*}
	and so, suppressing the dot notation,
	\[
	R_{\B'}R_{\B}(x,y)= R_{\B'}(x^{-1}(x\circ y), \overline{x^{-1}(x\circ y)}\circ x \circ y). \]
	\\
	The first component of this composition is therefore
\[\left(\left(x^{-1}(x\circ y)\right)\circ \left(\overline{x^{-1}(x\circ y)}\circ x \circ y\right)\right)\left(x^{-1}(x\circ y)\right)^{-1}=(x\circ y) (x\circ y)^{-1} x = x,\]
while the second component, using the reduction above, is
\[\overline{x}\circ x \circ y = y,\]
as required. 
\end{proof}
\begin{example}
	Return to the solution $R_{\B}$ from Example \ref{threeYBE}, namely 
	\[R(\eta^i\pi^j,\eta^k\pi^{\ell})=\left(\eta^{(-1)^jk+2i\ell+2j\ell}\pi^{\ell},\eta^{i-2j\ell}\pi^{j} \right),\]
	which was obtained from the brace in Example \ref{three}. The reader can check that we have 
	\[R_{\B'}(\eta^i\pi^j,\eta^k\pi^{\ell})=\left(\eta^{k+2j\ell}\pi^{\ell},\eta^{(-1)^{\ell}i+2jk+2j\ell}\pi^{j} \right).\]
	To verify that $R_{\B'}=R_{\B}^{-1}$, we have
	\begin{align*}
	R_{\B'}R_{\B}(\eta^i\pi^j) &= R_{\B'}\left(\eta^{(-1)^jk+2i\ell+2j\ell}\pi^{\ell},\eta^{i-2j\ell}\pi^{j} \right)\\
	&=\left(\eta^{i+2j{\ell}-2j\ell}\pi^j,\eta^{(-1)^j\left[(-1)^jk+2i\ell + 2j\ell \right]+2\ell(i-2j\ell)-2j\ell}\pi^{\ell} \right)\\
	&=\left(\eta^i, \eta^{k+(-1)^j\left[2i\ell + 2j\ell \right]+2i\ell-2j\ell}\pi^{\ell}\right)\\
	&=(\eta^i\pi^j,\eta^k\pi^{\ell})
	\end{align*} 
	since $\eta^4=1_B$. That $R_{\B'}R_{\B}(\eta^i\pi^j)=(\eta^i,\pi^j)$ is similar.
\end{example}

The explicit inverse solution allows us to identify group-like elements in the corresponding Hopf algebra. Recall that $h\in H$ is {\it group-like} if $\Delta(h)=h\otimes h$ where $\Delta$ is the comultiplication in the Hopf algebra $H$.

\begin{corollary}
	Let the Galois extension $L/K$ be $H$-Hopf Galois for some $K$-Hopf algebra $H_N$. Let $\B=(B,\cdot,\circ)$ be the brace corresponding to this Hopf-Galois structure, and for $i=1,2$ let $\pr_i:B\times B\to B$ be the projection onto the second factor. Then each $y\in B$ with $\pr_2 R_{\B}(x,y) = x$ for all $x$ naturally identifies with a group-like element of $H_N$, and vice-versa.
\end{corollary}

\begin{proof}
	 We claim that an element $h\in H_N=L[N]^G$ is group-like if and only if $h\in N$ and $G$ acts trivially upon it, that is, if and only if $h\in N\cap \rho(G)$. Indeed, $h\in H_N$ is group-like if and only if it is group-like when base changed to $L\otimes_K L[N]^G\cong L[N]$, and since the group-likes in $L[N]$ are the elements of the group $N$ it follows that $h$ is group-like if and only if $h\in N$, say $h=\eta\in N$. But $G$ acts trivially on $\eta$ if and only if $\lambda(g)\eta\lambda(g^{-1}) = \eta$ for all $g\in G$, i.e., $\eta\in\operatorname{Cent}_{\Perm(G)}(\lambda(G)) = \rho(G)$.
	
	Recall that $\B$ induces a regular, $(B,\circ)$ stable subgroup of $\Perm(B,\circ)$: $N=\{\eta_y: y\in B\}\le\Perm(B,\circ)$ where $\eta_y[z]=y\cdot z$, and $^x\eta_y= \eta_{(x\circ y)x^{-1}}$. So $(B,\circ)$ acts trivially on $\eta_y$ if any only if $(x\circ y)x^{-1}=y$, i.e., $\pr_1{R_{\B'}}(x,y)=y$ for all $x\in B$. This can only happen if $\pr_2 R_{\B}(x,y)=x$ since $R_{\B}R_{\B'}(x,y)=(x,y)$. Through the isomorphism $(B,\circ)\to G$ we obtain the grouplike in $H_N$.
\end{proof}
 \begin{example}
 	The trivial brace, corresponding to $N=\lambda(G)$, gives the solution  $R(x,y) = (y,y^{-1}xy)$. So $y$ is a group-like if and only if $y^{-1}xy=x$ for all $x\in B$, i.e., $y\in Z(B,\cdot)$.
 \end{example}
 
 \begin{example}
 	The almost trivial brace, corresponding to $N=\rho(G)$ and the classical Galois structure, gives the solution $R(x,y) = (x^{-1}yx,x)$. Clearly, every $y$ is group-like.
 \end{example}
 
 \begin{example} The brace considered in Example \ref{three}, corresponding to the Hopf-Galois structure in Example \ref{threeHGS}, gives the solution 
 	\[R(\eta^i\pi^j,\eta^k\pi^{\ell})=\left(\eta^{(-1)^jk+2i\ell+2j\ell}\pi^{\ell},\eta^{i+2j\ell}\pi^{j} \right).\]
 	One can see that $\pr_2R(\eta^i\pi^j,\eta^k\pi^{\ell})=\eta^i\pi^j$ for all $i,j$ if and only if $\ell$ is even, hence the group-likes correspond are elements of the form $\eta^k$. This makes sense since $\eta=\rho(s)$.
 \end{example}

\section{On the Hopf-Galois Correspondence}\label{HGS}

Suppose $L/K$ is Galois with Galois group $G$.  We have seen that any $N\le\Perm(G)$ regular, $G$-stable gives rise to a Hopf-Galois structure on $L/K$, but the correspondence between sub-Hopf algebras and intermediate fields is not surjective. It is natural to ask: which intermediate fields arise as the fixed field of a sub-Hopf algebra? Since the correspondence from sub-Hopf algebras to intermediate fields {\it is} injective, this is equivalent to determining the sub-Hopf algebras of $H_N$. 

\begin{definition}
	Let $L/K$ be Hopf-Galois for some Hopf algebra $H$. We say that an intermediate field $K\subseteq F \subseteq L$ is {\it realizable with respect to $H$} (or simply {\it realizable} for short) if $F=L^{H_0}$ for $H_0$ some sub-Hopf algebra of $H$.
\end{definition}

In \cite{Childs18}, Childs shows that realizable subfields are in one-to-one correspondence with what he calls ``$\circ$-stable (or `circle-stable') subgroups'' of the corresponding brace. For $\B=\B(N)=(B,\cdot,\circ)$, a subgroup $C\le (B,\cdot)$ is said to be {\it $\circ$-stable} if $(x\circ y)x^{-1}\in C$ for $x\in B, y\in C$. A $\circ$-stable subgroup is closed under $\circ$ as well, hence is a sub-brace of $\B$.

We will take a different approach to realizable subfields using the results of \cite{KochKohlTrumanUnderwood19a} and the concept of opposites. It is not hard to show that a $\circ$-stable subgroup, when viewed in the opposite brace, looks like the more familiar brace structure called an ideal, one which we generalize somewhat below by relaxing normality conditions.

\begin{definition}
	Let $\B=(B,\cdot,\circ)$ be a brace. 
	\begin{enumerate}
		\item A {\it quasi-ideal}  of $\B$ is a subgroup $I\le (B,\cdot)$ such that 
		\[ x^{-1}(x\circ y)\in I,\; x\in B, y\in I. \]
		\item A {\it $\cdot$-quasi-ideal} ($\cdot$-QI) of $\B$ is a quasi ideal which is normal in $(B,\cdot)$.
		\item A {\it $\circ$-quasi-ideal} ($\circ$-QI) of $\B$ is a quasi ideal which is normal in $(B,\circ)$.
		\item An {\it ideal} of $\B$ is a subgroup of $(B,\cdot)$ which is both a  $\cdot$-QI and a $\circ$-QI.
	\end{enumerate}
\end{definition}

Note that a quasi-ideal $I$ is also a subgroup of $(B,\circ)$, hence is a sub-brace of $\B$. To see this, note that for all $x,y,z\in B$ we have 
\[x^{-1}(x\circ y \circ z) = x^{-1}(x\circ y)(x\circ y)^{-1}(x\circ y\circ z),\]
and if $y,z\in I$ then $x^{-1}(x\circ y) \in I$ and $(x\circ y)^{-1}(x\circ y\circ z)\in I$, hence their product is in $I$ and $I$ is closed under $\circ$. Additionally,
\[1_B=x^{-1}(x\circ y\overline{y}) = x^{-1}(x\circ y) x^{-1}(x\circ \overline{y},\]
and since $1_B\in I$ and $x^{-1}(x\circ y)\in I$ we get that $x^{-1}(x\circ \overline{y})\in I$, i.e., $\overline{y}\in I$.

By \cite[Example 2.2]{GuarnieriVendramin17}, the kernel of a brace morphism has the structure of an ideal. Additionally, by \cite[Lemma 2.3]{GuarnieriVendramin17}, if $I$ is an ideal of $\B$ then both $I$ and $B/I$ are braces. Thus, ideals are essential to understanding the category of braces.

Now suppose $\B=\B(N)$ for $N$ a regular $G$-stable subgroup of $\Perm(G)$, where $G:=\Gal(L/K)$. Each of the substructures above gives us insight as to the intermediate fields in the $H_{N'}$-Hopf Galois structure on $L/K$, where as above $N'=\operatorname{Cent}_{\Perm(G)}(G)$ as before. 

We begin with the simplest of the structures.

\begin{lemma}
	Quasi-ideals $I$ of $\B:=\B(N)$ correspond bijectively with intermediate fields $K\le L_I \le L$ realizable with respect to $H_{N'}$. 
\end{lemma}
\begin{proof}
	Let $I$ be a quasi-ideal of $\B$. Since the underlying sets of $\B$ and $\B'$ are the same, namely $N$, and $x^{-1}(x\circ y) = (x\circ y)\cdot' x^{-1}$ for all $x,y\in N$ we get that $I$ is a $\circ$-stable subgroup of $\B'$. Through the isomorphism $(\B(N))'\to\B(N')$ its image is a $\circ$-stable subgroup in $\B(N')$, say $I'$. Then, by \cite[Theorem 4.3]{Childs18}, $I'$ corresponds to a sub-Hopf algebra of $H_{N'}$, hence an intermediate field in $L/K$ which is realizable with respect to $H_{N'}$. Conversely, if $F$ is a field which is realizable with respect to $H_{N'}$, there is a corresponding $\circ$-stable subgroup of $\B(N')$, hence of $\B'$, giving us a quasi-ideal of $\B$.
\end{proof}

	 By \cite[Prop. 2.2]{KochKohlTrumanUnderwood19a} (which itself is a reformulation of the ideas from \cite[\S 5]{GreitherPareigis87}), sub-Hopf algebras of $H_N$ correspond bijectively to $G$-stable subgroups $I$ of $N$, hence realizable fields correspond to such $I$. We can relate this theory to \cite{Childs18} as follows. Suppose $\B=(B,\cdot,\circ)$ is a brace, and $I$ is a $\circ$-stable subgroup of $\B$. Let $G=(B,\circ)$, and let $N=\{\eta_x:x\in B\}\le\Perm(G)$ where $\eta_x[y]=x\cdot y$. Let $I_*=\{\eta_{i}:i\in I\}\le N$. Then 
	 \[^x\eta_i=\eta_{(x\circ i)x^{-1}}, \;x,i\in B\]
	 and since $I$ is $\circ$-stable we know $(x\circ i)x^{-1}\in I$, hence $I_*$ is $G$-stable. It is easy to see that the converse holds as well.

	Additionally, if $I\trianglelefteq N$ then $L^{H_I}/K$ is also Hopf-Galois for a particular Hopf algebra related to $H_{N}$--see \cite[Theorem 2.10]{KochKohlTrumanUnderwood19a}. Thus we get:

\begin{lemma}
	There is a bijection between $\cdot$-quasi-ideals $I$ of $\B:=\B(N)$ and intermediate fields $K\le L_I \le L$  realizable with respect to $H_{N'}$ such that $L_I/K$ is also Hopf-Galois via the $K$-Hopf algebra $L_I[N'/I']^{G}$, where $I'$ is the image of $I$ under the isomorphism $I\mapsto I'$ above.
\end{lemma}

How $L_I[N'/I']^{G}$ acts on $L_I$ is not obvious--see the discussion prior to \cite[Theorem 2.10]{KochKohlTrumanUnderwood19a} for a complete description.


 Of course, if $I\trianglelefteq (B,\circ)$, then the corresponding subgroup of $G$ is also normal. This gives:
 
 \begin{lemma}
 	$\circ$-quasi-ideals $I$ of $\B:=\B(N)$ correspond bijectively with intermediate fields $K\le L_I \le L$  realizable with respect to $H_{N'}$ such that $L_I/K$ Galois. 
 \end{lemma}
 
 We summarize:

\begin{theorem}\label{Realizable}
	Let $L/K$ be Galois, group $G$, and let $N\le\Perm(G)$ be regular and $G$-stable. Let $\B=\B(N)$ and $\B'=(\B(N))'=\B(N')$. Let $I\subseteq \B$ be a quasi-ideal.
	Then there exists a field $K\le L_I\le L$ such that $L/L_I$ is Hopf-Galois via the $L_I$-Hopf algebra $L[I]^{G}$. Furthermore:
	\begin{enumerate}
		\item If $I$ is a $\cdot$-QI then $L_I/K$ is also Hopf-Galois with respect to a Hopf algebra which depends on $H$.
		\item If $I$ is a $\circ$-QI then $L_I/K$ is (classically) Galois.
		\item If $I$ is an ideal, then $L_I/K$ is both Galois and Hopf-Galois in the sense mentioned above.
	\end{enumerate} 
	Furthermore, any realizable intermediate field $F$ is of the form $L_I$ for some quasi-ideal $I$; and if $F$ satisfies the the proprieties (1), (2), or (3) above, then $I$ is a $\cdot$-QI, $\circ$-QI, or an ideal respectively.
\end{theorem}

\begin{example}
	Suppose $\B=(B,\cdot,\cdot)$ is the trivial brace.  If $I\le(B,\cdot)$ is any subgroup, then $I$ is automatically a quasi-ideal since $x^{-1}(x\circ y) = y$. It is an ideal if and only if $I$ is normal in $(B,\cdot)$. This makes sense since $\B'$ is (isomorphic to) the brace corresponding to the classical Galois structure: each subgroup gives an intermediate field, and the Hopf-Galois structure on $L_I$ coincides with the Galois structure when $I$ is normal.
\end{example}

\begin{example}
	Suppose $\B=(B,\cdot',\cdot)$ is the almost trivial brace.  If $I\le(B,\cdot)$ is any subgroup, then $I$ is a quasi-ideal if and only if $x^{-1}\cdot'(x\circ y) = xyx^{-1}$ for all $x\in B, y\in I$, i.e., if $I$ is normal in $(B,\cdot)$. If this is the case, then it is automatically an ideal as well. 
\end{example}

\begin{example}
	Let $B=\gen{\eta,\pi}\cong D_4$ with, as usual,
	\[\eta^i\pi^j\circ \eta^k\pi^{\ell} = \eta^{k+(-1)^{\ell}i+2j{\ell}}\pi^{j+\ell},\;0\le i,k \le 3,\; 0\le j,\ell\le 1.\]
	Of course, $I=\{1_B\}$ and $I=B$ are ideals. The group $(B,\cdot)$ has eight other subgroups.
	\begin{description}
		\item[$I=\gen{\eta}$] We have $(x_j)^{-1}(x_j\circ \eta^k) = x_j^{-1}\eta^kx_j\in I$ since $I$ is normal in $(B,\cdot)$. Thus $I$ is a quasi-ideal. Since $I$ is also a subgroup of $(B,\circ)$ it must be normal in this group as well, hence $I$ is an ideal.
		\item[$I=\gen{\eta^2}$] This must be a quasi-ideal as well from the work above, as well as an ideal since $I=Z(B,\cdot)=Z(B,\circ)$. 
		\item[$I=\gen{\eta^k\pi},\;0\le k\le 3$] Since
		\[(\eta\pi)^{-1}(\eta\pi\circ \eta^k\pi) = \eta\pi(\eta^{k-1}) = \eta^{2-k}\pi \not\in\gen{\eta^k\pi}\]
		we see that $I$ is not a quasi-ideal.
		\item[$I=\gen{\eta^2,\pi}$] From the above, the quasi-ideal condition $x^{-1}(x\circ y)\in I$ holds for $y=\eta^2$. For $k=0,2$ we have
		\[(\eta^{i}\pi^{j})^{-1}(\eta^i\pi^j\circ \eta^k\pi)=\pi^{-j}\eta^{-i}(\eta^{i+k}\pi^{j+1})=\eta^k\pi\in I \]
		so $I$ is a quasi-ideal of $\B$. It is also an ideal since $I$=4.
		\item[$I=\gen{\eta^2,\eta\pi}$] For $k=1,3$ we have
		\[(\eta^{i}\pi^{j})^{-1}(\eta^i\pi^j\circ \eta^k\pi)=\pi^{-j}\eta^{-i}(\eta^{i-k}\pi^{j+1})=\eta^{(-1)^jk}\pi\in I \]
		and is also an ideal.
	\end{description}
\end{example}

\section{Self-Opposite Braces}\label{SO}

We conclude this paper with a discussion concerning  self-opposite braces. Of course, $\B=\B'$ if and only if $(B,\cdot)$ is an abelian group. However, it is possible for $\B$ and $\B'$ to be isomorphic, as the following example shows.

\begin{example}
	Let $(G,\ast_G)$ be any nonabelian group. Let $B=G\times G$, and define two operations on $B$ as follows:
	\[
	(x,y)\cdot(z,w) = (x,y)\circ(z,w) = (x\ast_Gz,y\ast_Gw).
	\]
	It is easy to show that $\B:=(B,\cdot,\circ)$ is a brace, and that the map $T:\B\to\B'$ given by $T(x,y)=(y,x)$ is an isomorphism.
\end{example}

More generally, if $\B$ is any brace, then so is $\B\times \B'$, and $(\B\times \B')'=(\B'\times \B) \cong \B\times \B'$. While equality, not isomorphism, is required for $R_{\B}$ and $R_{\B'}$ to be equal, the enumeration of realizable fields depends only on the isomorphism class of $\B$. Clearly, if $\B(N)$ is self-opposite then quasi-ideal, etc.\@ classify the realizable, etc.\@, fields in the sense of Theorem \ref{Realizable} corresponding to the Hopf algebra $H_N$. 

Because of this, it would be interesting to have sufficient, and possibly necessary, conditions for a brace to be self-opposite. While we do not have a full set of conditions (though certainly $(B,\cdot)$ abelian, or $\B\cong \mathfrak{C}\times \mathfrak{C}'$ suffice), we do have some necessary conditions, from which we can determine some braces which are not self-opposite. 

For example, let us call $(x,y)\in B\times B$ an {\it L-pair} if $x\circ y = xy$; if $x\circ y = yx$ then we will call $(x,y)$ an {\it R-pair}. If $\phi:\B\to\B'$ is an isomorphism and $(x,y)$ is an L-pair of $\B$, then
\[\phi(x)\circ \phi(y)= \phi(x\circ y) = \phi(x\cdot y) = \phi(x)\cdot'\phi(y) = \phi(y)\cdot \phi(x),\]
and hence $(\phi(x),\phi(y))$ is an R-pair of $\B$. Thus we get:

\begin{proposition}
	If the number of L-pairs and R-pairs of $\B$ is not equal, then $\B$ is not self-opposite.
\end{proposition}

\begin{example}
	Let us consider Example \ref{three} one last time: $B=\gen{\eta,\pi}\cong D_4$ with
	\[\eta^i\pi^j\circ \eta^k\pi^{\ell} = \eta^{2j\ell}(\eta^k\pi^{\ell})(\eta^i\pi^j)= \eta^{k+(-1)^{\ell}i+2j{\ell}}\pi^{j+\ell},\;0\le i,k \le 3,\; 0\le j,\ell\le 1.\]
	If $\eta^i\pi^j\circ \eta^k\pi^{\ell}=\eta^i\pi^j\eta^k\pi^{\ell}=\eta^{i+(-1)^jk}\pi^{j+\ell}$ then we must have
	\[k+(-1)^{\ell}i+2j{\ell}\equiv i+(-1)^jk \pmod 4.\]
	Picking $j=\ell=0$ gives us $16$ L-pairs. If $j=1,\ell=0$ we get
	\[k+i\equiv i-k \pmod 4,\]
	which provides $8$ pairs, corresponding to the cases $k=0,2$. Setting $j=0,\ell=1$ gives another $8$ pairs, and if $j=\ell=1$ we get
		\[k-i+2\equiv i-k \pmod 4,\]
	so $2(i+k)\equiv 2 \pmod{4}$, which holds if $i$ and $k$ are of different parity, giving another $8$ pairs. In total, $\B$ has $40$ L-pairs.
	
	On the other hand, if  $\eta^i\pi^j\circ \eta^k\pi^{\ell} = \eta^{2j\ell}(\eta^k\pi^{\ell})(\eta^i\pi^j)= \eta^k\pi^{\ell}\eta^i\pi^j,\;j,\ell=0,1$ then it is necessary and sufficient that $2j\ell= 0$, in other words either $j=0$ or $\ell=0$ or both. This gives $48$ R-pairs for $\B$, hence $\B$ is not self-opposite.
\end{example}

\bibliographystyle{amsalpha}
\bibliography{MyRefs}

\end{document}